\documentclass[12pt,reqno]{amsart}
\usepackage{amsmath,amsthm,amssymb,amsfonts,amscd}
\usepackage{mathrsfs}
\usepackage{bbm}
\usepackage{bbding}
\usepackage{hyperref}
\usepackage{geometry}
\usepackage{color}
\usepackage{xcolor}

\usepackage{picture,epic}
\usepackage{tikz}

\usepackage{enumitem}
\usepackage{fancyhdr}

\numberwithin{equation}{section}

\theoremstyle{plain}
\newtheorem{theorem}{Theorem}[section]
\newtheorem{lemma}[theorem]{Lemma}

\theoremstyle{definition}

\theoremstyle{remark}

\makeatletter
\def\@tocline#1#2#3#4#5#6#7{\relax
  \ifnum #1>\c@tocdepth 
  \else
    \par \addpenalty\@secpenalty\addvspace{#2}%
    \begingroup \hyphenpenalty\@M
    \@ifempty{#4}{%
      \@tempdima\csname r@tocindent\number#1\endcsname\relax
    }{%
      \@tempdima#4\relax
    }%
    \parindent\z@ \leftskip#3\relax \advance\leftskip\@tempdima\relax
    \rightskip\@pnumwidth plus4em \parfillskip-\@pnumwidth
    #5\leavevmode\hskip-\@tempdima
      \ifcase #1
       \or\or \hskip 1em \or \hskip 2em \else \hskip 3em \fi%
      #6\nobreak\relax
    \hfill\hbox to\@pnumwidth{\@tocpagenum{#7}}\par
    \nobreak
    \endgroup
  \fi}
\makeatother

\begin{document}

\title[On the $r$-free values of the polynomial $x^2 + y^2 + z^2 +k$]
{On the $r$-free values of the polynomial $x^2 + y^2 + z^2 +k$}
\author{Gongrui Chen,\,\,\,\, Wenxiao Wang*}
\address{School of Mathematics, Shandong University \\
Jinan, Shandong 250100, China}
\email{cgr4258@gmail.com(G.-R. Chen),wxwang@mail.sdu.edu.cn(W.-X. Wang)}

\begin{abstract}
 Let $k$ be a fixed integer. We study the asymptotic formula of $R(H,r,k)$, which is the number of positive integer solutions $1\leq x, y,z\leq H$ such that the polynomial $x^2+y^2+z^2+k$ is $r$-free. We obtained the asymptotic formula of $R(H,r,k)$ for all $r\ge 2$. Our result is new even in the case $r=2$. We proved that $R(H,2,k)= c_kH^3 +O(H^{9/4+\varepsilon})$, where $c_k>0$ is a constant depending on $k$. This improves upon the error term $O(H^{7/3+\varepsilon})$ obtained by Zhou and Ding \cite{ZD}.
\end{abstract}
\thanks{* Corresponding author}
\thanks{This work is supported by NSFC grant 11922113.}
\maketitle

\section{Introduction}

\indent

There exists an outstanding conjecture that the polynomial $x^{2}+1$ contains infinitely many primes. Iwaniec \cite{Iwa} proved there are infinitely many $n$ such that $n^{2}+1$ has at most two prime factors. So far, various authors considered the square-free values of some specific polynomials with integer coefficients. In 1931, Estermann \cite{Est} showed that
$$\sum_{1\leq x\leq H}\mu^{2}(x^{2}+1)=c_{0}H+O(H^{2/3+\varepsilon}),$$
where $c_{0}$ is an absolute constant. This error term was improved to $O(H^{7/12+\varepsilon})$ by Heath-Brown \cite{HB1}. Carlitz \cite{Car} studied  the polynomial $x^{2}+x$ and established that
$$\sum_{1\leq x\leq H}\mu^{2}(x^{2}+x)=\prod_{p}(1-\frac{2}{p^{2}})H+O(H^{2/3+\varepsilon}).$$
Later, Heath-Brown \cite{HB2} improved the error term $O(H^{2/3+\varepsilon})$ to $O(H^{7/11}\log^{7}H)$ which was further improved by Reuss \cite{Reuss} to $O(H^{0.578+\varepsilon})$. Mirsky \cite{Mir2} studied the square-free values of the polynomial $(x+a_{1})(x+a_{2})$.

In 2012 Tolev\cite{Tolev} considered the square-free values of a polynomial in two variables. Let $S_{2}(H)$ stand for the number of the square-free values of $x^{2}+y^{2}+1$ with $1\leq x,y \leq H$. Using Weil's estimate for the Kloosterman sum, Tolev proved that
$$S_{2}(H)=\prod_{p}(1-\frac{\lambda_{2}(p^{2})}{p^{4}})H^{2}+O(H^{4/3+\varepsilon}),$$
where $\lambda_{2}(q)$ is the number of the integer solutions to the following congruence equation
$$x^{2}+y^{2}+1\equiv 0 ({\rm{mod}}\ q),\quad 1\leq x,y\leq q.$$
In 2022 Zhou and Ding \cite{ZD} studied the square-free values of the polynomial $x^{2}+y^{2}+z^{2}+k$, where $k$ is a fixed (nonzero) integer. Let \begin{equation}\label{defineS3}
S_{3}(H,k)=\sum_{\substack{1\leq x,y,z \leq H}}\mu^2(x^{2}+y^{2}+z^{2}+k).
\end{equation} It was proved in \cite{ZD} that
\begin{equation}\label{eq1.2}
S_{3}(H,k)=\prod_{p}(1-\frac{\lambda(p^{2};k)}{p^{6}})H^{3}+O(H^{7/3+\varepsilon}),
\end{equation}
where $\lambda(q;k)$ is defined as
\begin{equation}\label{definelambda}
\lambda(q;k)=\sum_{\substack{1\leq x,y,z \leq q \\ x^{2}+y^{2}+z^{2}+k\equiv 0 ({\rm{mod}}\ q)}}1.
\end{equation}

In addition, some articles are devoted to the $r$-free values of polynomials in one variable. Note that for $r\ge 2$, an integer $n$ is $r$-free if $p^r\nmid n$ for all primes $p$. Let $S(H,r)$ be the number of integers $x\leq H$ such that $x^2+x$ is $r$-free. In 1932, Carlitz \cite{Car} obtained
$$S(H,r)=\prod_{p}(1-\frac{2}{p^{r}})H+O(H^{2/(r+1)+\varepsilon}).$$
This above error term was improved to $O(H^{14/(7r+8)+\varepsilon})$ by Brandes \cite{Brandes} generalizing Heath-Brown's method in \cite{HB2}. Brandes' result was further improved by Reuss \cite{Reuss}. Furthermore, Reuss \cite{Reuss} established a more general result
$$S_{r,h}(H)=c_{r,h}H+O(H^{\upsilon(r)+\varepsilon}),$$
where $S_{r,h}(H)$ stands for the number of integers $x\leq H$ such that $x$ and $x+h$ are $r$-free, $c_{r,h}$ is a  constant depending on $r$ and $h$, and $\upsilon(2)=0.578$ while $\upsilon(r)=169/(144r)$ for $r\geq 3$. Also, the $r$-free values of the polynomial $(x+a_{1})(x+a_{2})$ was considered by Mirsky \cite{Mir1} who gave the error term $O(H^{2/(r+1)+\varepsilon})$.

Inspired by the above works, we study the $r$-free values of $x^2+y^2+z^2+k$. We define
\begin{equation}\label{defineR}
R(H,r,k)=\sum_{\substack{1\leq x,y,z \leq H \\ x^{2}+y^{2}+z^{2}+k \textrm{ is } r\textrm{-free}}}1.
\end{equation}
The main result of this paper is to obtain the asymptotic formula of $R(H,r,k)$.
\begin{theorem}\label{th1}Let $r\ge 2$ be a natural number. Let $k$ be a fixed integer. Let $\varepsilon>0$ be an arbitrary small positive number. We have the following results.

(i) For $k\neq0$,
\begin{equation}\label{Rk1}
R(H,r,k) = \prod_p\left(1-\frac{\lambda(p^r;k)}{p^{3r}}\right)H^3 +O(H^{2}+H^{3/2+3/2r+\varepsilon}).
\end{equation}

(ii)
\begin{equation}\label{R01}
R(H,r,0) = \prod_p\left(1-\frac{\lambda(p^r;0)}{p^{3r}}\right)H^3 +O(H^{2+\varepsilon}+H^{9r/(5r-2)+\varepsilon}).
\end{equation}
\end{theorem}

Note that $R(H,2,k)$ is exactly $S_{3}(H,k)$ defined in \eqref{defineS3}. In fact, Theorem \ref{th1} is also new in the case $r=2$. We have the following.
\begin{theorem}\label{th2}Let $k$ be a fixed integer. Let $\varepsilon>0$ be an arbitrary small positive number. We have
$$S_{3}(H,k)= \prod_p\left(1-\frac{\lambda(p^2;k)}{p^{6}}\right)H^3 +O(H^{9/4+\varepsilon}).$$
\end{theorem}
Theorem \ref{th2} improves upon \eqref{eq1.2} obtained by Zhou and Ding \cite{ZD}.

\bigskip

\section{Notations and some lemmas}

Let $H$ be a sufficiently large positive number.
The letters $k, m, n, l, a, b,c $ stand for integers and $ d, r, h, q, x, y, z, \alpha, \rho$ stand for positive integers.
The letters $\eta,\xi$ denote real numbers and the letter $p$ is reserved for primes.
By $\varepsilon$ we denote an arbitrary small positive number. Throughout this paper, $k$ and $r\ge 2$ are fixed integers, and the implied constants may depend on $k,r$ and $\varepsilon$.

As usual, the functions $\mu(n)$ and $\tau(n)$ represent the M\"obius function and the number of positive divisors of $n$, respectively.
We write $(n_1, \dots, n_u)$ for the greatest common divisor of $n_1, \dots, n_u$.
Let $\| \xi \|$ be the distance from $\xi$ to its nearest integer.
Further $e(t) = \exp \left( 2 \pi i t \right) $ and $e_q(t) = e(t/q)$.
For any $q$ and $x$ such that $(q, x)=1$ we denote by $\overline{x}_q$
the inverse of $x$ modulo $q$. If the modulus $q$ is clear form the context then we write for simplicity $\overline{x}$. For any odd $q$ we denote by $\left( \frac{\cdot}{q}\right) $ the Jacobi symbol.

We introduce the Gauss sum
\begin{equation} \label{1.1}
  G(q; n, m) = \sum_{1 \le x \le q} e_q \left( n x^2 + m x\right) , \quad
  G(q; n) = \sum_{1 \le x \le q} e_q \left( n x^2 \right).
\end{equation}

First, we introduce some basic properties of the Gauss sum.

\begin{lemma} \label{Gauss}
For the Gauss sum we have

(i)
If $(q, n)=d$ then
\[
  G(q; n, m) = \begin{cases}
      d \, G (q/d; \, n/d, \, m/d) & \text{if} \quad d \mid m , \\
      0 & \text{if} \quad d \nmid m .
        \end{cases}
\]

(ii)
If $(q, 2n) = 1 $ then
\[
  G(q; n, m) = e_q \left( - \overline{(4n)} \, m^2 \right) \,
  \left( \frac{n}{q} \right) \, G(q; 1) .
\]

(iii)
If $(q, 2)=1$ then
\[
    G^2(q; 1) = (-1)^{\frac{q-1}{2}} \, q .
\]
\end{lemma}

We introduce the Sali\'{e} sum
\begin{equation} \label{defineSalie}
  S(c; a, b) = \sum_{\substack{ 1 \le x \le c \\ (x, c) = 1 }}\left(\frac{x}{c}\right) e_c \left( a x + b \overline{x} \right) .
\end{equation}
It is easy to see that
\begin{equation} \label{sym}
  S(c; a, b) = S(c;b,a) .
\end{equation}
The following result comes from Corollary 4.10 in \cite{IB}.

\begin{lemma} \label{s1}
Let $p$ be an odd prime. Let $\alpha$  be a positive integer. If $p\nmid a$ (or $p\nmid b$), then
\begin{equation} \label{2.1}
 | S(p^\alpha;\, a,\, b)|\leq \tau(p^\alpha)p^{\frac{\alpha}{2}}.
\end{equation}
\end{lemma}

We define
\begin{equation} \label{1.2}
  T(c; a,\rho) = \sum_{\substack{ 1 \le x \le c \\ (x, c) = 1 }}\left(\frac{x}{c}\right)^{\rho} e_c \left( a x \right) .
\end{equation}

\begin{lemma} \label{s3}
Let $p$ be an odd prime. Let $\alpha$  and $\rho$ be  positive integers. Suppose that $p\nmid a$. One has
\[
 T(p^{\alpha};\, a,\, \rho)= 0\quad \textrm{ for } \ \alpha\geq2,
\]
and
\[
 |T(p;\, a,\, \rho)|\leq p^{\frac{1}{2}}.
\]
\end{lemma}

\begin{proof}
We first consider the case $\alpha=1$. Note that $T(p; a;\rho)$ is either the Gauss sum or the Ramanujan sum, and we have
$|T(p;\, a,\, \rho)|\leq p^{\frac{1}{2}}$.

Now we consider the case $\alpha\geq 2$. We write $x=yp+z$ to deduce that
\begin{equation*}
\begin{aligned}T(p^{\alpha};\, a,\, \rho)
&=\sum_{0\leq y\leq p^{\alpha-1}-1}\sum_{0 \leq z \leq p-1}\left(\frac{z}{p}\right)^{\rho}e_{p^{\alpha}}(apy+az)\\
&=\sum_{0 \leq z \leq p-1}\left(\frac{z}{p}\right)^{\rho}e_{p^{\alpha}}(az)\sum_{0\leq y\leq p^{\alpha-1}-1}e_{p^{\alpha-1}}(ay)=0.
\end{aligned}
\end{equation*}
This completes the proof.
\end{proof}

\begin{lemma} \label{Salie0}
Let $p$ be an odd prime and let $\alpha$ be a positive integer. One has
\begin{align}\label{boundS}
  |S(p^\alpha; a,0)| \le p^{\frac{\alpha}{2}} \, (p^\alpha, a)^{\frac{1}{2}} .
\end{align}
\end{lemma}
\begin{proof}Note that \eqref{boundS} holds trivially when $p^\alpha|a$. We only need to consider $p^\alpha\nmid a$. We assume that $p^\beta\|a$ and $0\leq \beta\leq \alpha-1$. Then we write $a=p^\beta a'$ with $p\nmid a'$. We have
$$S(p^\alpha;\, a,\, 0)=\sum_{\substack{1 \leq x \leq p^{\alpha}\\(x,p)=1}}\left(\frac{x}{p}\right)^{\alpha}e_{p^{\alpha-\beta}}(a'x).$$
We write $x=yp^{\alpha-\beta}+z$ to deduce that
\begin{equation*}
\begin{aligned}
 S(p^\alpha;\, a,\, 0)
&=\sum_{0\leq y\leq p^{\beta}-1}\sum_{\substack{0 \leq z \leq p^{\alpha-\beta}-1\\(z,p)=1}}\left(\frac{z}{p}\right)^{\alpha}e_{p^{\alpha-\beta}}(a'z)\\
&=p^{\beta}\sum_{\substack{0 \leq z \leq p^{\alpha-\beta}-1\\(z,p)=1}}\left(\frac{z}{p}\right)^{\alpha}e_{p^{\alpha-\beta}}(a'z)
\\&=p^{\beta}T(p^{\alpha-\beta};\, a',\, \alpha).
\end{aligned}
\end{equation*}
Since $p\nmid a'$, we conclude from Lemma \ref{s3} that
\begin{equation*}
\begin{aligned}
 |S(p^\alpha;\, a,\, 0)|\le p^{\beta+\frac{\alpha-\beta}{2}}= p^{\frac{\alpha}{2}+\frac{\beta}{2}}=p^{\frac{\alpha}{2}}(p^\alpha,a)^{1/2}.
\end{aligned}
\end{equation*}
This completes the proof.
\end{proof}

Let
\begin{equation} \label{1.3}
  \lambda(q; n, m, l, k) =  \sum_{\substack{ 1 \le x, \, y, \, z \le q \\ x^2 + y^2 + z^2 + k \equiv 0 ({\rm{mod}}\ q) }}
   e_q \left( n x + m y + lz \right) .
\end{equation}
\begin{lemma}[Lemma 2.2 \cite{ZD}]\label{lemmamul}
Suppose that $(q_1,q_2)=1$. One has
$$\lambda(q_1q_2;n,m,l,k)=\lambda(q_1; \overline{q_2}_{q_1}n, \overline{q_2}_{q_1}m,\overline{q_2}_{q_1}l,k) \lambda(q_2; \overline{q_1}_{q_2}n,\overline{q_1}_{q_2}m,\overline{q_1}_{q_2}l,k).$$
In particular, we have
$$\lambda(q_1q_2;k)=\lambda(q_1;k)\lambda(q_2;k).$$
\end{lemma}

Now we apply Lemma \ref{s1}, Lemma~\ref{Salie0} and Lemma~\ref{lemmamul} to obtain an upper bound of $\lambda(q;n,m,l,k)$.
\begin{lemma}\label{lamda}
Suppose that $ p^{r}\| q$ for all primes $p|q$.

(i) If $k\neq 0$, then we have
\begin{align}\label{boundlam1}\lambda (q;n,m,l,k) \ll q^{1+\varepsilon} (q,n,m,l).\end{align}

(ii) One has (for $k=0$)
\begin{align}\label{boundlam2}\lambda(q; n, m, l, 0) \ll q^{1+\varepsilon}(q, n, m, l)(q, n^2+m^2+l^2)^{1/2}.\end{align}
\end{lemma}
\begin{proof}
We write $q=p_1^{r}p_2^{r} \dots p_s^{r}$, where $p_{i}(1\le i\le s)$ are distinct primes. By Lemma~\ref{lemmamul}, we have
\begin{equation}\label{lammi}
\lambda(q;n,m,l,k)=\prod_{i=1}^s \lambda\left(p_i^{r}; n\overline{q_i}, m\overline{q_i}, l\overline{q_i}, k\right),
\end{equation}
where $q_i=q/p_{i}^{r}$ and $\overline{q_i}$ stands for the inverse of $q_i$ modulo $p_i^{r}$.

Let $n_i=n\overline{q_i}$, $m_i=m\overline{q_i}$ and $l_i=l\overline{q_i}.$ Note that $(p_i^{r},\overline{q_i})=1$, so we have $(p_i^{r},n_i,m_i,l_i)=(p_i^{r},n,m,l)$ and $(p_i^{r},n_i^2+m_i^2+l_i^2)=(p_i^{r},n^2+m^2+l^2)$. Therefore, we only need to prove
\begin{align}\label{boundlam3}\lambda (p^r;u,v,w,k) \ll p^{r+\varepsilon} (p^r,u,v,w) \ \ \textrm{ for } k\not=0,\end{align}
and
\begin{align}\label{boundlam4}\lambda(p^r; u,v,w, 0) \ll p^{r+\varepsilon}(p^r, u,v,w)(p^r, u^2+v^2+w^2)^{1/2}.\end{align}

One has the trivial bound $|\lambda (p^r;u,v,w,k)|\le p^{3r}$. Since the implied constants were allowed to depend on $r$ and $k$, we assume that
$p$ is odd and we further assume that $p\nmid k$ if $k\not=0$.

We have
\begin{equation*}
\begin{aligned}
\lambda(p^{r}; u, v, w,k) &=p^{-r}\sum_{1 \leq x,y,z \leq p^{r}}e_{p^{r}}(u x+v y+w z)\sum_{1 \leq h \leq p^{r}}e_{p^{r}}(h(x^2+y^2+z^2+k)) \\
&=p^{-r}\sum_{1 \leq h \leq p^{r}}e_{p^{r}}(hk) G(p^{r};h,u) G(p^{r};h,v) G(p^{r};h,w) \\
&=p^{-r}\sum_{0 \leq \beta \leq r} \sum_{\mbox{\tiny$\begin{array}{c} 1 \leq h \leq p^{r} \\ (h,p^{r})=p^{\beta} \end{array}$}}e_{p^{r}}(hk) G(p^{r};h,u) G(p^{r};h,v) G(p^{r};h,w).
\end{aligned}
\end{equation*}
 We write $h=h'p^{\beta}$ with $1 \leq h' \leq p^{r-\beta}$ and $(h', p^{r-\beta})=1$. Then we have
 \begin{equation}\label{lambapr}
\begin{aligned}
&\lambda(p^{r}; u, v, w,k)\\
&=p^{-r}\sum_{0 \leq \beta \leq r} \sum_{\mbox{\tiny$\begin{array}{c} 1 \leq h' \leq p^{r-\beta} \\ (h,p^{r-\beta})=1 \end{array}$}}e_{p^{r-\beta}}(h'k) G(p^{r};h'p^{\beta},u) G(p^{r};h'p^{\beta},v) G(p^{r};h'p^{\beta},w).
\end{aligned}
\end{equation}
 By Lemma~\ref{Gauss} (i), we have
  \begin{equation}\label{lambapr2}
\begin{aligned}&G(p^{r};h,u) G(p^{r};h,v) G(p^{r};h,w)
\\ & =
p^{3\beta}G(p^{r-\beta}; h',up^{-\beta})G(p^{r-\beta}; h', vp^{-\beta}) G(p^{r-\beta}; h', wp^{-\beta})
\end{aligned}
\end{equation}
if  $p^{\beta} \mid (u,v,w)$, and
$G(p^{r};h,u) G(p^{r};h,v) G(p^{r};h,w)=0$ if $p^{\beta} \nmid (u,v,w)$. When $p$ is odd and $p^{\beta} \mid (u,v,w)$, by Lemma \ref{Gauss} (ii), we have
\begin{equation}\label{lambapr3}
\begin{aligned}
&G(p^{r-\beta}; h',up^{-\beta})G(p^{r-\beta}; h', vp^{-\beta}) G(p^{r-\beta}; h', wp^{-\beta}) \\
&=\left(\frac{h'}{p^{r-\beta}}\right)^3 G(p^{r-\beta},1)^3e_{p^{r-\beta}} \left(-\overline{(4h')}_{p^{r-\beta}}(u^2+v^2+w^2)p^{-2\beta}\right).
\end{aligned}
\end{equation}
From \eqref{lambapr}-\eqref{lambapr3}, we obtain
\begin{equation}\label{lambapr4}
\begin{aligned}
\lambda(p^{r}; u, v, w,k) =p^{-r}\sum_{\mbox{\tiny$\begin{array}{c} 0\leq \beta \leq r  \\ p^{\beta} \mid (u,v,w) \end{array}$}} p^{3\beta}G(p^{r-\beta};1)^3
S\left(p^{r-\beta}; k, -\overline{4}\frac{u^2+v^2+w^2}{p^{2\beta}}\right).
\end{aligned}
\end{equation}

We first deal with the case $k\not=0$. As mentioned above, $p$ is odd and $p\nmid k$,
then by Lemma \ref{Gauss} (iii) and Lemma~\ref{s1}, we have
\begin{equation*}
\begin{aligned}
\lambda(p^{r}; u, v, w,k) &\ll p^{-r}\sum_{\mbox{\tiny$\begin{array}{c} 0\leq \beta \leq r  \\ p^{\beta} \mid ({p}^{r},u,v,w) \end{array}$}}p^{3\beta}p^{\frac{3}{2}(r-\beta)}\tau(p^{r-\beta})p^{\frac{1}{2}(r-\beta)}
  \ll  p^{r+\varepsilon}(p^{r},u,v,w).
\end{aligned}
\end{equation*}
This establishes \eqref{boundlam3}. Now we consider the case $k=0$. By  Lemma \ref{Gauss} (iii) and Lemma \ref{Salie0}, we deduce from \eqref{lambapr4} that
\begin{equation*}
\begin{aligned}
\lambda(p^{r}; u, v, w,0) &\ll p^{-r}\sum_{\mbox{\tiny$\begin{array}{c} 0\leq \beta \leq r  \\ p^{\beta} \mid ({p}^{r},u,v,w) \end{array}$}}p^{3\beta}p^{\frac{3}{2}(r-\beta)}p^{\frac{1}{2}(r-\beta)}\left(p^{r-\beta},\frac{u^2+v^2+w^2}{p^{2\beta}}\right)^{1/2}
\\ &\ll p^{r+\varepsilon}(p^r, u,v,w)(p^r, u^2+v^2+w^2)^{1/2}.
  \end{aligned}
\end{equation*}
This establishes \eqref{boundlam4}. The proof of the lemma is complete.
\end{proof}

\begin{lemma}[Lemma 2.5 \cite{ZD}]\label{uvw}
Let $Q\geq2$. Let $q=d^{r}$ with $d$ odd and square-free.
We define
\begin{align*}U_1(Q,q)=&\sum_{1\leq n \leq Q}\frac{|\lambda(q;n,0,0,k)|}{n},
\\ U_2(Q,q)=&\sum_{1\leq n,m \leq Q}\frac{|\lambda(q;n,m,0,k)|}{nm},
\\ U_3(Q,q)=&\sum_{1\leq n,m,l \leq Q}\frac{|\lambda(q;n,m,l,k)|}{nml}.
\end{align*}
Suppose that $k \neq 0$. For $1\le i\le 3$, we have
$$U_i(Q,q) \ll q^{1+\varepsilon}Q^{\varepsilon}.$$
\end{lemma}


\begin{lemma}[Lemma 4.7 \cite{Nath}]\label{eta}
 For any real number $\xi$ and all integers $N_{1},\, N_{2}$ with $N_{1}<N_{2}$,
$$\sum_{n=N_{1}+1}^{N_{2}}e(\xi n)\ll min\{N_{2}-N_{1}, \|\xi\|^{-1}\}.$$
\end{lemma}

Now we introduce
\begin{align}N_1(H,q,k)=&\frac{1}{q}\sum_{1\leq t\leq q-1}\lambda(q;-t,0,0,k)\sum_{1\leq h\leq H}e_{q}(ht),\label{defineN1}
\\ N_2(H,q,k)=&\frac{1}{q^2}\sum_{1\leq t_{1}, t_{2} \leq q-1}\lambda(q;-t_1,-t_2,0,k)\prod_{i=1}^{2}\Big(\sum_{1\leq h_i\leq H}e_{q}(h_{i}t_{i})\Big),\label{defineN2}
\\ N_3(H,q,k)=& \frac{1}{q^3}\sum_{1\leq t_{1}, t_{2}, t_{3} \leq q-1}\lambda(q;-t_1,-t_2,-t_3,k)\prod_{i=1}^{3}\Big(\sum_{1\leq h_i\leq H}e_{q}(h_{i}t_{i})\Big).\label{defineN3}\end{align}

\begin{lemma}\label{N}
Let $H \geq 2$. Let $d$ be odd and square-free. Suppose that $d^{r} \ll H^{2}$. If $k \neq 0$, then for $1\le i\le 3$ we have
$$N_i(H,d^{r},k) \ll d^{r+\varepsilon}.$$
\end{lemma}
\begin{proof}By Lemma~\ref{eta}, we have
$$\sum_{1\leq h\leq H}e_{d^{r}}(ht) \ll \|\frac{t}{d^r}\|^{-1}.$$
Then we obtain
\begin{equation*}
\begin{aligned}
N_i(H,d^{r},k)\ll U_i(\frac{d^r-1}{2},d^r).
\end{aligned}
\end{equation*}
The desired estimate follows from Lemma~\ref{uvw} directly.
\end{proof}

\section{Proof of Theorem \ref{th1}}

Let
$$D(H,q,k)=\sum_{\mbox{\tiny$\begin{array}{c} 1\leq x,y,z \leq H \\ x^2+y^2+z^2+k\equiv 0 ({\rm{mod}}\ q)\end{array}$}}1.$$
\begin{lemma}\label{lemmaboundD}Suppose that $q\ll H^2$. Then we have
\begin{align}\label{boundD}D(H,q,k)\ll q^{-1}H^{3+\varepsilon}.\end{align}
In particular, one has
\begin{align}\label{boundlambdaq}\lambda(q;k)\ll q^{2+\varepsilon}.\end{align}
\end{lemma}
\begin{proof}The proof is standard and we include the details for completeness. We have
\begin{equation*}
\begin{aligned}
D(H,q,k) =&\sum_{|t| \leq \frac{3H^2+|k|}{q}}\sum_{1\leq z \leq H} \sum_{\mbox{\tiny$\begin{array}{c} 1\leq x,y \leq H \\ x^2+y^2=qt-z^2-k \end{array}$}}1\\ \ll
&\sum_{|t| \leq \frac{3H^2+|k|}{q}}\sum_{1\leq z \leq H}H^{\varepsilon}
\\ \ll &q^{-1}H^{3+\varepsilon}.
\end{aligned}
\end{equation*}
This establishes \eqref{boundD}. Note that $\lambda(q;k)=D(q,q,k)$. The estimate \eqref{boundlambdaq} follows immediately from \eqref{boundD}.
This completes the proof.\end{proof}

Now we represent $R(H,r,k)$ in terms of $D(H,q,k)$.
\begin{lemma}\label{lemma42}Suppose that $1\le \eta\le (3H^2+|k|)^{1/r}$. Then we have
$$R(H,r,k)=\sum_{d\le \eta}\mu(d)D(H,d^r,k)+O(H^{3+\varepsilon}\eta^{1-r}).$$
\end{lemma}
\begin{proof}
The starting point is to apply the following identity
\begin{align*}
\sum_{d^r|n}\mu(d)=\begin{cases}1, \ &\textrm{ if } n \textrm{ is } r\textrm{-free},
\\ 0, \ &\textrm{ if } n\not=0 \textrm{ is not }r\textrm{-free}.\end{cases}
\end{align*}
Then we deduce that
\begin{equation*}
\begin{aligned}
R(H,r,k) &= \sum_{\mbox{\tiny$\begin{array}{c} 1 \leq x,y,z \leq H\\x^2+y^2+z^2+k \neq 0 \end{array}$}}
\sum_{\mbox{\tiny$\begin{array}{c} d^r \mid x^2+y^2+z^2+k  \end{array}$}}\mu(d) \\
&=\sum_{1\leq d\leq (3H^2+|k|)^{1/r}}\mu(d)\sum_{\mbox{\tiny$\begin{array}{c} 1\leq x,y,z \leq H \\ x^2+y^2+z^2+k\equiv 0 ({\rm{mod}}\ d^r)\\ x^2+y^2+z^2+k\not=0\end{array}$}}1
.
\end{aligned}
\end{equation*}
Note that
$$\sum_{\mbox{\tiny$\begin{array}{c} 1\leq x,y,z \leq H \\ x^2+y^2+z^2+k\equiv 0 ({\rm{mod}}\ q)\\ x^2+y^2+z^2+k\not=0\end{array}$}}1
=D(H,q,k)+O(1).$$
Now we conclude from above that
\begin{equation*}
\begin{aligned}
R(H,r,k) =&\sum_{1\leq d\leq (3H^2+|k|)^{1/r}}\mu(d)D(H,d^r,k)+\sum_{1\leq d\leq (3H^2+|k|)^{1/r}}O(1)
\\= &\sum_{1\leq d\leq (3H^2+|k|)^{1/r}}\mu(d)D(H,d^r,k)+O(H^{2/r}).
\end{aligned}
\end{equation*}
Splitting the above summation into two parts, we obtain
\begin{equation*}
\begin{aligned}
R(H,r,k)=\sum_{1\leq d\leq \eta}\mu(d)D(H,d^r,k)+\sum_{\eta<d\leq (3H^2+|k|)^{1/r}}\mu(d)D(H,d^r,k)+O(H^{2/r}).
\end{aligned}
\end{equation*}
On applying Lemma \ref{lemmaboundD}, we further obtain
$$R(H,r,k)=\sum_{1\leq d\leq \eta}\mu(d)D(H,d^r,k)+O(H^{3+\varepsilon}\eta^{1-r}).$$
This completes the proof.
\end{proof}

Now we deal with $D(H,q,k)$ for $q\le \eta^r$.
\begin{lemma}\label{lemma43}Let $\lambda(q;k)$ be defined in \eqref{definelambda} and let $N_i(H,q,k)$ be defined in \eqref{defineN1}-\eqref{defineN3}. We have
\begin{equation}\label{Ds}
D(H,q,k)=\frac{H^3}{q^{3}}\lambda(q;k)+3\frac{H^2}{q^{2}}N_1(H,q,k) + 3\frac{H}{q}N_2(H,q,k) +N_3(H,q,k).
\end{equation}
\end{lemma}
\begin{proof}Note that
\begin{equation*}
\begin{aligned}
D(H,q,k)=\sum_{\mbox{\tiny$\begin{array}{c} 1\leq x,y,z \leq q \\ x^{2}+y^{2}+z^{2}+k \equiv 0 ({\rm{mod}}\ q) \end{array}$}}
\sum_{\mbox{\tiny $\begin{array}{c} 1\leq m,n,l \leq H \\ m \equiv x ({\rm{mod}}\ q) \\ n \equiv y ({\rm{mod}}\ q)\\ l \equiv z ({\rm{mod}}\ q)\end{array}$}} 1.
\end{aligned}
\end{equation*}
By orthogonality, we have
\begin{equation*}
\begin{aligned}
\sum_{\mbox{\tiny $\begin{array}{c} 1\leq m \leq H \\ m \equiv x ({\rm{mod}}\ q) \end{array}$}} 1
= & q^{-1}\sum_{1\leq m \leq H} \sum_{1\leq t \leq q}e_{q}((h-x)t)\\
=& q^{-1}\sum_{1\leq t \leq q}e_{q}(-xt)\sum_{1\leq h \leq H}e_{q}(ht)\\
=& Hq^{-1}+q^{-1}\sum_{1\leq t \leq q-1}e_{q}(-xt)\sum_{1\leq h \leq H}e_{q}(ht).
\end{aligned}
\end{equation*}
Now we conclude from above that
\begin{equation}\label{DandL}
\begin{aligned}
D(H,q,k)=\sum_{\mbox{\tiny$\begin{array}{c} 1\leq x,y,z \leq q \\ x^{2}+y^{2}+z^{2}+k \equiv 0({\rm{mod}}\ q) \end{array}$}}\Big(\frac{H^{3}}{q^{3}}+3\frac{H^{2}}{q^{2}}L_{1}(x)
+3\frac{H}{q}L_{2}(x,y)+L_{3}(x,y,z)\Big),
\end{aligned}
\end{equation}
where
$$L_1(x):=L_{1}(x;q,H)=\frac{1}{q}\sum_{1\leq t \leq q-1}e_{q}(-xt)\sum_{1\leq h \leq H}e_{q}(ht),$$
$$L_{2}(x,y):=L_{2}(x,y;q,H)=\frac{1}{q^2}\sum_{1\leq t_{1},t_{2} \leq q-1}e_{q}(-xt_{1}-yt_{2})\prod_{i=1}^{2}(\sum_{1\leq h_{i} \leq H}e_{q}(h_{i}t_{i})),$$
and
$$L_{3}(x,y,z):=L_{3}(x,y,z;q,H)=\frac{1}{q^3}\sum_{1\leq t_{1},t_{2},t_{3} \leq q-1}e_{q}(-xt_{1}-yt_{2}-zt_{3})\prod_{i=1}^{3}(\sum_{1\leq h_{i} \leq H}e_{q}(h_{i}t_{i})).$$
For convience, we write $L_1(x,y,z;q,H)=L_1(x;q,H)$ and $L_2(x,y,z;q,H)=L_2(x,y;q,H)$. By exchanging the order of summations and recalling the definitions of $\lambda(q;n,m,l,k)$ and $N_i(H,q,k)$, we obtain
\begin{equation}\label{LandN}
\begin{aligned}
\sum_{\mbox{\tiny$\begin{array}{c} 1\leq x,y,z \leq q \\ x^{2}+y^{2}+z^{2} \equiv 0 ({\rm{mod}}\ q) \end{array}$}}L_{i}(x,y,z;q,H)=
N_i(H,q,k).
\end{aligned}
\end{equation}
Now \eqref{Ds} follows from \eqref{DandL} and \eqref{LandN}.
\end{proof}

\noindent\textit{Proof of Theorem \ref{th1}.}
We first deal with the case $k\not=0$. By Lemma~\ref{N} and Lemma \ref{lemma43}, we can get
$$D(H,d^r,k)=H^3\frac{\lambda(d^r;k)}{d^{3r}}+O(H^2d^{\varepsilon-r}+Hd^{\varepsilon}+d^{r+\varepsilon}).$$
By Lemma \ref{lemma42}, we find that
$$R(H,r,k)= H^3\sum_{1\leq d \leq \eta}\frac{\mu(d)\lambda(d^r;k)}{d^{3r}}+O(H^{2}+H\eta^{1+\varepsilon} +\eta^{r+1+\varepsilon}+H^{3+\varepsilon}\eta^{1-r}),$$
and then by \eqref{boundlambdaq}, we have
$$R(H,r,k)= H^3\sum_{d=1}^{\infty}\frac{\mu(d)\lambda(d^r;k)}{d^{3r}}+O(H^{2}+H\eta^{1+\varepsilon} +\eta^{r+1+\varepsilon}+H^{3+\varepsilon}\eta^{1-r}).$$
Now we choose $\eta=H^{3/2r}$ to conclude that
$$R(H,r,k)=H^3\sum_{d=1}^{\infty}\frac{\mu(d)\lambda(d^r;k)}{d^{3r}} +O(H^{2}+H^{3/2+3/2r+\varepsilon}).$$
Since $\lambda(q;k)$ is multiplicative as a function of $q$, we obtain (for $k\not=0$) that
$$R(H,r,k)=\prod_p\left(1-\frac{\lambda(p^r;k)}{p^{3r}}\right)H^3 +O(H^{2}+H^{3/2+3/2r+\varepsilon}),$$
and this completes the proof of \eqref{Rk1}.

From now on, we consider the case $k=0$. We introduce
$$T_{1}(H,\eta)=\sum_{1\leq d\leq \eta}\frac{\mu(d)}{d^{2r}}N_1(H,d^r,0),$$
$$T_{2}(H,\eta)=\sum_{1\leq d\leq \eta}\frac{\mu(d)}{d^{r}}N_2(H,d^r,0),$$
and
$$T_{3}(H,\eta)=\sum_{1\leq d\leq \eta}\mu(d)N_3(H,d^r,0).$$
By Lemma \ref{lemma42} and \eqref{Ds}, we have
\begin{equation}\label{RTO}
\begin{aligned}
&R(H,r,0)\\
&=H^3\sum_{1\leq d \leq \eta}\frac{\mu(d)\lambda(d^r;0)}{d^{3r}}+3H^2T_1(H,\eta)+3HT_2(H,\eta)+T_{3}(H,\eta)+O(H^{3+\varepsilon}\eta^{1-r}).
\end{aligned}
\end{equation}
Recalling the definition of $N_1(H,d^r,0)$ in \eqref{defineN1}, we deduce from \eqref{boundlam2} and Lemma~\ref{eta} that
\begin{equation*}
\begin{aligned}
T_{1}(H,\eta)
&\ll \sum_{1\leq d\leq \eta}d^{-2r}\sum_{1 \leq |t| \leq \frac{d^r-1}{2}} \frac{|\lambda(d^r;t,0,0,0)|}{|t|} \\
&\ll \sum_{1\leq d\leq \eta}d^{-r+\varepsilon}\sum_{1 \leq |t| \leq \frac{d^r-1}{2}}  \frac{(d^r,t)(d^r,t^2)^{\frac{1}{2}}}{|t|}.
\end{aligned}
\end{equation*}
Since $(d^r,t^2)^{\frac{1}{2}}\le d^{r/2}$, we have
\begin{equation*}
\begin{aligned}
T_{1}(H,\eta)\ll \sum_{1\leq d\leq \eta}d^{-\frac{r}{2}+\varepsilon}\sum_{1 \leq |t| \leq \frac{d^r-1}{2}}  \frac{(d^r,t)}{t}.
\end{aligned}
\end{equation*}
From the following elementary estimate
$$\sum_{1 \leq t \leq \frac{d^r-1}{2}} \frac{(d^r,t)}{t}\ll d^{\varepsilon},$$
we conclude that (for $r\ge 2$)
\begin{align}\label{boundT1}T_{1}(H,\eta)
\ll \sum_{1\leq d\leq \eta}d^{-\frac{r}{2}+\varepsilon}\ll \eta^{\varepsilon}.
\end{align}
Recalling \eqref{defineN2}, we deduce from \eqref{boundlam2} and Lemma~\ref{eta} that
\begin{equation*}
\begin{aligned}
T_{2}(H,\eta)
& \ll \sum_{1\leq d\leq \eta}d^{-r}\sum_{1 \leq |t_1|,|t_2| \leq \frac{d^r-1}{2}} \frac{|\lambda(d^r;t_{1},t_{2},0,0)|}{|t_{1}t_{2}|}\\
&\ll \eta^\varepsilon\sum_{1\leq d\leq \eta}\sum_{1 \leq |t_1|,|t_2| \leq \frac{d^r-1}{2}} \frac{(d^r,t_{1}^2+t_{2}^2)^{1/2}(d^r,t_{1},t_{2})}{|t_{1}t_{2}|}.
\end{aligned}
\end{equation*}
Since $(d^r,t_{1}^2+t_{2}^2)^{1/2}\leq d^{\frac{r-2}{2}}(d,t_{1}^2+t_{2}^2)\le \eta^{\frac{r-2}{2}}(d,t_{1}^2+t_{2}^2)$ and $(d^r,t_{1},t_{2})\leq (t_{1},t_{2})$, we deduce that
$$T_{2}(H,\eta)
\ll \eta^{\frac{r-2}{2}+\varepsilon}\sum_{1\leq d\leq \eta}\sum_{1 \leq |t_1|,|t_2| \leq \frac{d^r-1}{2}}(d,t_{1}^2+t_{2}^2)\frac{(t_{1},t_{2})}{t_{1}t_{2}}.$$
By exchanging the order of summations, we have
\begin{align}\label{exchangeT2}T_{2}(H,\eta)
\ll \eta^{\frac{r-2}{2}+\varepsilon}\sum_{1 \leq t_1,t_2 \leq \eta^t}\frac{(t_{1},t_{2})}{t_{1}t_{2}}\sum_{1\leq d\leq \eta}(d,t_{1}^2+t_{2}^2).
\end{align}
Now we easily obtain
\begin{align}\label{boundT2}T_{2}(H,\eta)
\ll  \eta^{\frac{r}{2}+\varepsilon}.
\end{align}
We estimate the sum $T_{3}(H,\eta)$ in same way as follows
\begin{equation*}
\begin{aligned}
T_{3}(H,\eta)
 \ll \sum_{1\leq d\leq \eta}d^{r+\varepsilon}\sum_{1 \leq |t_{1}|,|t_{2}|,|t_{3}| \leq \frac{d^r-1}{2}} \frac{(d^r,t_{1}^2+t_{2}^2+t_{3}^2)^{1/2}(d^r,t_{1},t_{2},t_{3})}{|t_{1}t_{2}t_{3}|},
\end{aligned}
\end{equation*}
and by $(d^r,t_{1}^2+t_{2}^2+t_{3}^2)^{1/2}\leq d^{\frac{r-2}{2}}(d,t_{1}^2+t_{2}^2+t_{3}^2)\leq \eta^{\frac{r-2}{2}}(d,t_{1}^2+t_{2}^2+t_{3}^2)$ and $(d^r,t_{1},t_{2},t_{3})\leq (t_{1},t_{2},t_{3})$, we have
$$T_{3}(H,\eta)
\ll \eta^{\frac{3r-2}{2}+\varepsilon}\sum_{1\leq d\leq \eta}\sum_{1 \leq |t_{1}|,|t_{2}|,|t_{3}| \leq \frac{d^r-1}{2}}(d,t_{1}^2+t_{2}^2+t_{3}^2)\frac{(t_{1},t_{2},t_{3})}{t_{1}t_{2}t_{3}}.$$
By a similar argument as in \eqref{exchangeT2}, we finally obtain
\begin{align}\label{boundT3}
T_{3}(H,\eta)\ll \eta^{3r/2+\varepsilon}.\end{align}
Now we combine \eqref{RTO}-\eqref{boundT3} to conclude that
$$R(H,r,0) = H^3\sum_{1\leq d\leq \eta}\frac{\mu(d)\lambda(d^r;0)}{d^{3r}} +O(H^{2+\varepsilon}+H\eta^{\frac{r}{2}+\varepsilon}+\eta^{\frac{3r}{2}+\varepsilon}+H^{3+\varepsilon}\eta^{1-r}).$$
Then by \eqref{boundlambdaq}, we get
$$R(H,r,0) = H^3\sum_{d=1}^\infty\frac{\mu(d)\lambda(d^r;0)}{d^{3r}} +O(H^{2+\varepsilon}+H\eta^{\frac{r}{2}+\varepsilon}+\eta^{\frac{3r}{2}+\varepsilon}+H^{3+\varepsilon}\eta^{1-r}).$$
Since $\lambda(q;0)$ is multiplicative, we obtain by choosing $\eta =H^{6/(5r-2)}$ that
$$R(H,r,0) = \prod_p\left(1-\frac{\lambda(p^r;0)}{p^{3r}}\right)H^3 +O(H^{2+\varepsilon}+H^{9r/(5r-2)+\varepsilon}).$$
This establishes \eqref{R01}. The proof of Theorem \ref{th1} is complete.

\end{document}